\documentclass[graybox]{svmult}

\usepackage{mathptmx}       
\usepackage{helvet}         
\usepackage{courier}        
\usepackage{type1cm}        
\usepackage{makeidx}         
\usepackage{graphicx}        
\usepackage{tikz}
\usetikzlibrary{arrows}
\usetikzlibrary{shapes,shadows,calc,positioning,shadows,backgrounds,petri}
\usepackage{multicol}        
\usepackage[bottom]{footmisc}
\usepackage{bm}

\makeindex             

\usepackage{latexsym}
\usepackage{dsfont}
\usepackage{bbm}
\usepackage{amssymb}
\usepackage{amsmath}

\newtheorem{Theorem}{Theorem}[section]
\newtheorem{Lemma}[Theorem]{Lemma}

\newtheorem{Prop}[Theorem]{Proposition}

\DeclareMathAlphabet\mathbfcal{OMS}{cmsy}{b}{n} 

\def\sfA{\mathsf{A}}

\def\sfB{\mathsf{B}}

\def\sfv{\mathsf{v}}

\def\Ext{\mathsf{Ext}}
\def\Surv{\mathsf{Surv}}

\def\N{\mathbb{N}}
\def\V{\mathbb{V}}

\def\P{\mathbb{P}}
\def\E{\mathbb{E}}

\def\T{\mathbb{T}}
\def\calT{\mathcal{T}}
\def\calbfT{\mathbfcal{T}}

\def\calZ{\mathcal{Z}}

\def\calL{\mathcal{L}}

\def\1{\mathds{1}}

\def\bfN{\mathbf{N}}

\def\eps{\varepsilon}

\def\calTproc{(\calT_{n})_{n\ge 0}}
\def\calTprocStar{(\calT_{n}^{*})_{n\ge 0}}
\def\calZproc{(\calZ_{n})_{n\ge 0}}

\def\bfTprocStar{(\calbfT_{n}^{*})_{n\ge 0}}
\def\Xfam{(X^{(\bullet,k)}_{i,\sfv})_{k\ge1,i\ge1,\sfv\in\V}}
\def\Tfam{(N_{\sfv})_{\sfv\in\V}}

\def\assBPRE{(Z_{n}')_{n\ge 0}}

\allowdisplaybreaks[4] 

\begin{document}

\title*{Branching within branching I: The extinction problem}
\titlerunning{Branching within branching I: The extinction problem}
\author{Gerold Alsmeyer and S\"oren Gr\"ottrup}
\institute{Gerold Alsmeyer and S\"oren Gr\"ottrup\at Inst.~Math.~Statistics, Department
of Mathematics and Computer Science, University of M\"unster,
Orl\'eans-Ring 10, D-48149 M\"unster, Germany.\at
\email{gerolda@math.uni-muenster.de, soeren.groettrup@gmail.com}}

\maketitle

\abstract{We consider a discrete-time host-parasite model for a population of cells which are colonized by proliferating parasites. The cell population grows like an ordinary Galton-Watson process, but in reflection of real biological settings the multiplication mechanisms of cells and parasites are allowed to obey some dependence structure. More precisely, the number of offspring produced by a mother cell determines the reproduction law of a parasite living in this cell and also the way the parasite offspring is shared into the daughter cells. In this article, we provide a formal introduction of this branching-within-branching model and then focus on the property of parasite extinction. We establish equivalent conditions for almost sure extinction of parasites, and find a strong relation of this event to the behavior of parasite multiplication along a randomly chosen cell line through the cell tree, which forms a branching process in random environment. In a second paper \cite{AlsGroettrup:15b}, the case when parasites survive is studied by proving limit results.}

\bigskip



\section{Introduction}

The discrete-time \emph{branching-within-branching process (BwBP)} studied in this paper describes the evolution of generations of a population of cells containing proliferating parasites. 
In an informal way, its reproduction mechanism may be described as follows: 
\begin{description}[xxx]
\item[(1)] At time $n=0$ there is just one cell containing one parasite. 
\item[(2)] Cells and their hosted parasites within one generation form independent reproduction units which behave independently and in the same manner.
\item[(3)] Any cell splits into a random number $N$, say, of daughter cells in accordance with a probability distribution $(p_k)_{k\ge 0}$. 
\item[(4)] Then, given $N$, the hosted parasites, independently and in accordance with the same distribution, produce random numbers of offspring which are then shared into the daughter cells.
\item[(5)] All cells and parasites obtained from a cell and its parasites in generation $n$ belong to generation $n+1$.
\end{description}
We are thus dealing with a hierarchical model of two subpopulations, viz. cells and parasites, with an entangled reproduction mechanism. The hierarchy stems from the fact that cells can survive without parasites but not vice versa.

\vspace{.1cm}
Proceeding with a more formal introduction, let $\V$ denote the infinite Ulam-Harris tree with root $\varnothing$ and $N_{\sfv}$ the number of daughter cells of cell $\sf\sfv\in\V$. The $(N_{\sfv})_{\sf\sfv\in\V}$ are independent and identically distributed (iid) copies of the $\N_{0}$-valued random variable $N$ with distribution $(p_k)_{k\ge 0}$ and finite mean $\nu$, viz. $\P(N=k)=p_{k}$ for all $k\in\N_{0}$ and 
$$ \nu=\E N<\infty. $$
The cell population thus forms a standard \emph{Galton-Watson tree (GWT)} $\T=\bigcup_{n\in\N_{0}}\T_{n}$ with $\T_{0}=\{\varnothing\}$ and
\begin{equation*}
\T_{n} := \{\sfv_{1}\dots\sfv_{n}\in\V|\sfv_{1}\dots\sfv_{n-1}\in\T_{n-1}\ \text{and}\ 1\le\sfv_{n}\le N_{\sfv_{1}\dots\sfv_{n-1}}\}
\end{equation*}
$($using the common tree notation $\sfv_{1}...\sfv_{n}$ for $(\sfv_{1},...,\sfv_{n}))$.
Consequently, defining
\begin{equation}\label{Eq.calT}
\calT_{n}\ :=\ \#\T_{n}\ =\ \sum_{\sfv\in\T_{n-1}}N_{\sfv}
\end{equation}
as the \emph{number of cells in the $n^{th}$ generation} for $n\in\N_{0}$, the sequence $\calTproc$ forms a standard \emph{Galton-Watson process (GWP)} with reproduction law $(p_k)_{k\ge 0}$ and reproduction mean $\nu$. For basic information on Galton-Watson processes see \cite{Athreya+Ney:72, Jagers:75}.

\vspace{.1cm}
Let $Z_{\sfv}$ denote the number of parasites in cell $\sfv\in\V$ and $\T_{n}^{*}$ the set of contaminated cells in generation $n\in\N_{0}$ with cardinal number $\calT_{n}^{*}$, so
\begin{equation}\label{Eq.contCell}
\T_{n}^{*}\ :=\ \{\sfv\in\T_{n}:Z_{\sfv}>0\}\quad\text{and}\quad\calT_{n}^{*}\ :=\ \#\T_{n}^{*}.
\end{equation}
We define the \emph{number of parasites process} by
\begin{equation*}
 \calZ_{n}\ :=\ \sum_{\sfv\in \T_{n}}Z_{\sfv},\quad n\in\N_{0}.
\end{equation*}
After these settings, the BwBP is defined as the pair 
$$ \left(\T_{n},(Z_{\sfv})_{\sfv\in\T_{n}}\right)_{n\ge 0} $$ 
and now clearly seen as a description of the generations of any given population of cells and the number of parasites hosted by them.
 
 \vspace{.1cm}
As informally stated above, parasites at different cells are assumed to multiply independently of each other, whereas reproduction of parasites living in the same cell $\sfv$ is conditionally iid given the number of daughter cells of $\sfv$. More precisely, let for each $k\in\N$
\begin{equation*}
X^{(\bullet,k)}_{i,\sfv}\ :=\ \left(X^{(1,k)}_{i,\sfv},\dots,X^{(k,k)}_{i,\sfv}\right),\quad i\in\N,\ \sfv\in\V,
\end{equation*}
be iid copies of an $\N_{0}^{k}$-valued random vector $X^{(\bullet,k)}:=\left(X^{(1,k)},\dots,X^{(k,k)}\right)$. The families $\big(X^{(\bullet,k)}_{i,\sfv}\big)_{i\in\N,\sfv\in\V}$, $k\in\N$, are assumed to be mutually independent and also independent of $\Tfam$. They provide the numerical description of 
the reproduction and sharing of the parasites living in the cell tree. In detail, if the cell $\sfv\in\V$ has $k\in\N$ daughter cells $\sfv 1,...,\sfv k$, then $X^{(j,k)}_{i,\sfv}$, $1\le j\le k$, gives the number of progeny of the $i^{\,th}$ parasite in cell $\sfv$ which go in daughter cell $\sfv j$. The sum over all entries in $X^{(\bullet,k)}_{i,\sfv}$ gives the total offspring number of this parasite.
The number of parasites in the cells are thus recursively defined by $Z_\varnothing:=1$ and
\begin{equation}\label{Eq.parasitesIncell}
Z_{\sfv j}\ :=\ \sum_{k\ge j}\1_{\{N_{\sfv}=k\}}\sum_{i=1}^{Z_{\sfv}}X^{(j,k)}_{i,\sfv}\ =\  \sum_{i=1}^{Z_{\sfv}}X^{(j,N_{\sfv})}_{i,\sfv},\qquad j\in\N,
\end{equation}
where, by convention, $X^{(j,k)}_{i,\sfv}:=0$ if $j>k$. We further define
\begin{equation*}
\mu_{j,k}\ :=\ \E X^{(j,k)}
\end{equation*}
for $j,k\in\N$ and
\begin{equation*}
\gamma\ :=\ \E\calZ_{1}\ =\ \sum_{k\ge 1}p_k\sum_{j=1}^{k}\mu_{j,k}
\end{equation*}
as the mean number of offspring per parasite, which is assumed to be positive and finite, i.e.
\begin{equation}\tag{A1}\label{As.Gamma}
0<\gamma<\infty.
\end{equation}
This naturally implies $\mu_{j,k}<\infty$ for all $j\le k$, and $\P(N=0)<1$. To avoid trivial cases, we further assume that
\begin{equation}\tag{A2}\label{As.Constant1}
p_{1}=\P(N=1)<1\quad\text{and}\quad\P(\calZ_{1}=1)<1.
\end{equation}
If the first assumption fails, the cell tree is just a cell line and $(\calZ_{n})_{n\ge 0}$ a standard GWP with reproduction law $\calL(X^{(1,1)})$, while failure of the second assumption entails the number of parasites in each generation to be the same, thus $\calT_{n}^{*}=\calT^{*}_{0}$ a.s. for all $n\in\N_{0}$, or $\calT_{n}^{*}=\calZ_{0}$ eventually.
To rule out the simple case that every contaminated daughter cell contains only one parasite (as the root cell $\varnothing$), we further assume that 
\begin{equation}\tag{A3}\label{As.neq1}
p_{k}\,\P(X^{(j,k)}\ge 2)>0\quad\text{for at least one $(j,k),\ 1\le j\le k$}.
\end{equation}

\begin{figure}[t]
\centering
\includegraphics[width=11cm]{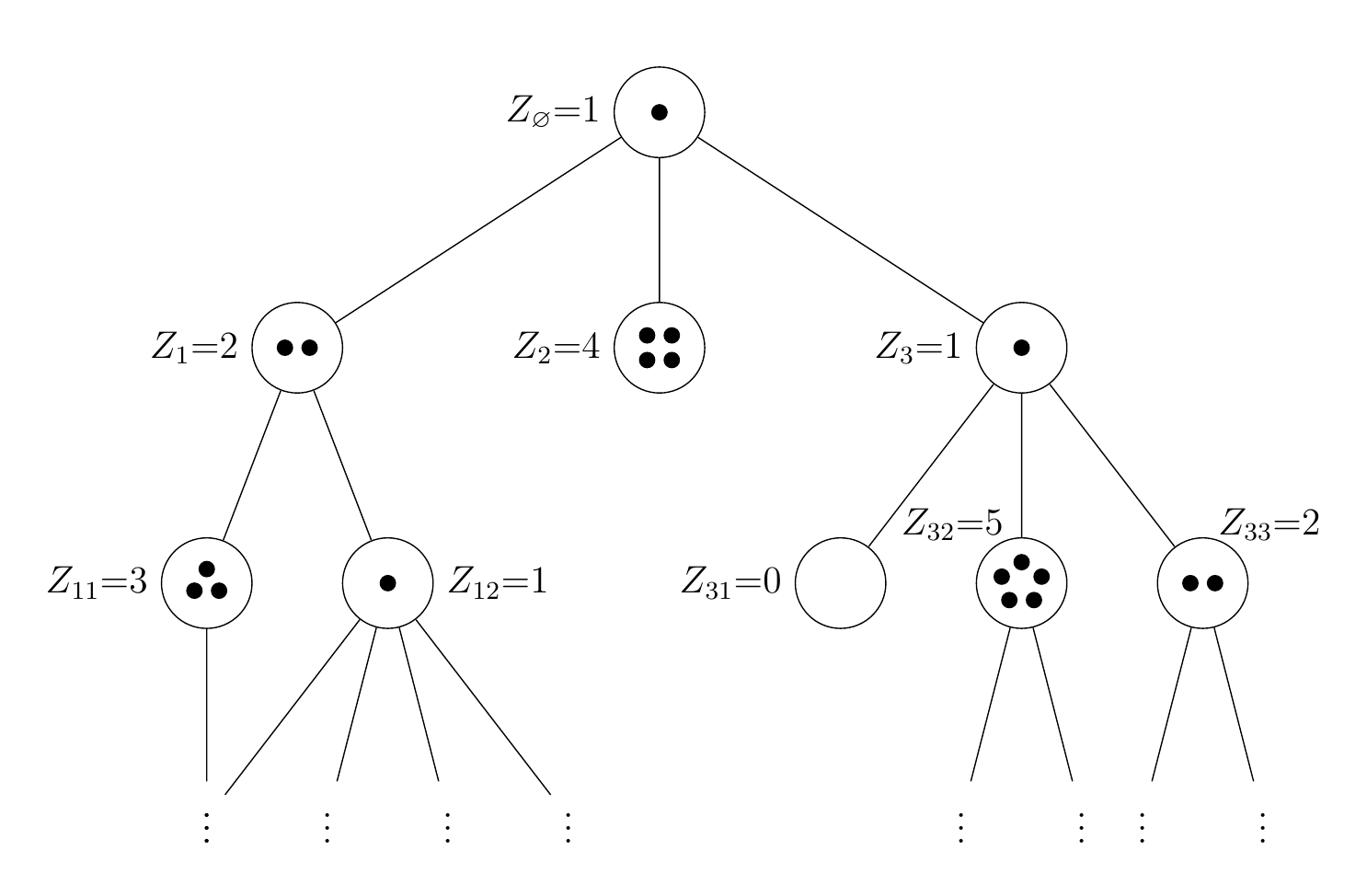}
\caption{A typical realization of the first three generations of a BwBP.}
\label{Fig.CellTree}
\end{figure}

For each generation $n=0,1,2,\ldots$, let
$$ \calbfT_{n} := \left(\calT_{n,0}, \calT_{n,1}, \calT_{n,2}, \dots\right) $$
denote the infinite vector \emph{of cell counts with a specific number of parasites}, so $\calT_{n,0}$ gives the number of non-infected cells and $\calT_{n,k}$ for $k\ge 1$ the number of cells with exactly $k$ parasites in generation $n$. Then $(\calbfT_{n})_{n\in\N_{0}}$ is a \emph{multi-type branching process (MTBP)} with countably many types. The individuals are the cells and the type of each cell is given by the number of parasites it contains. 

In situations where our BwBP initially has $s_{k}$ cells of type $k$ for $k\in\N_{0}$ and thus a total number $\sum_{i\ge 0}s_{i}$ of cells, we write $\P_{\vec{s}}$ for the underlying probability measure, where
\begin{equation}\label{def of bfN}
\vec{s}=(s_{0},s_{1},\dots)\in\bfN:=\big\{(x_{i})_{i\ge 0}\in\N_{0}^{\infty}|x_{i}>0\ \text{finitely often}\big\}.
\end{equation}
Further, we write $\P_{z}$, $z\in\N_{0}$ if the process starts with one cell and $z$ parasites, i.e.
\begin{equation*}
\P_{z}(\calT_{0}=1, Z_\varnothing=z)=1.
\end{equation*}
The corresponding expectations are denotes as usual by $\E_{\vec{s}}$ and $\E_{z}$, respectively. Indices are dropped in the standard case, viz. $\P=\P_{1}$ and $\E=\E_{1}$.

\vspace{.1cm}
MTBP's with finite type-space are well-studied with results transferred from the classical theory of GWP's, see \cite[Chapter V]{Athreya+Ney:72}, \cite[Chapter 4]{Jagers:75} or the monography by Mode \cite{Mode:71a}. If, on the other hand, the state space is infinite (countable or uncountable), a variety of behaviors may occur depending on the reproduction mechanism of types. For example, the branching random walk, an object of long standing interest \cite{AddarioReed:09, Aidekon:13, Biggins:77a, Biggins:77, Biggins:79, Biggins:92, BigKyp:97,HuShi:09,Kingman:75}, is obtained when type-space transitions are of independent additive kind. MTBP's are also studied in \cite{Athreya:00, BigKyp:04, Kesten:89, MenshikovVolkov:97, Moy:67, SagitovSerra:09,Serra:06}, and we refer to Kimmel and Axelrod \cite[Chapter 7]{Kimmel+Axelrod:02} for a series of examples with applications in biology.

\vspace{.1cm}
Various examples of branching-within-branching models have already been studied in the literature, the first one presumably by Kimmel \cite{Kimmel:97} who considered binary cell division in continuous time with symmetric sharing of parasites into the daughter cells. A discrete-time version of his model, and in fact a special case of ours, with possible asymmetric sharing of parasites was studied by Bansaye \cite{Bansaye:08} and later extended in \cite{Bansaye:09} by adding immigration of parasites and random environments. The latter means that parasites in a cell reproduce under the same but randomly chosen distribution. Subsequently, efforts have been made to generalize the underlying binary cell tree of the afore-mentioned models to arbitrary Galton-Watson trees in both, discrete and continuous time. The greatest progress in this direction has been achieved by Delmas and Marsalle in \cite{DelmasMar:10} and together with Bansaye and Tran in \cite{BanDelMarTran:11}. Both articles consider a random cell splitting mechanism and asymmetric sharing but make ergodic hypotheses that rule out the possibility of parasite extinction which is the focus in the present article. Let us further mention work by Guyon \cite{Guyon:07} on another discrete-time model with asymmetric sharing and by Bansaye and Tran \cite{BansayeTran:11} on a bifurcating cell-division model in continuous time with parasite evolution following a Feller diffusion. There, the cell division rates depend on the quantities of parasites inside the cells and asymmetric sharing of parasites into the two daughter cells is assumed. 

\vspace{.1cm}
Beside the model in \cite{Bansaye:08}, our model also comprises the one of type-$\sfA$ cells studied in \cite{AlsGroettrup:13}, where these cells produce daughter cells of either the same type $\sfA$ or of type $\sfB$, and the sharing mechanism of parasites in a type-$\sfA$ cell may depend on the number of type-$\sfA$ and type-$\sfB$ daughter cells. As another special case, we mention the situation of multinomial repartition of parasites. Here all parasites multiply independently in accordance with the same offspring distribution so that the number of parasites process $(\calZ_{n})_{n\ge 0}$ forms a standard GWP. After parasite reproduction, a cell divides into a random number $N$ of descendants, and the offspring of each of its hosted parasites chooses independently the $i^{\,th}$ daughter cell with probability $q_{i}(k)\in[0,1]$ if $N=k$. Thus,
\begin{equation}\label{eq:multinomial sharing}
 \sum_{j=1}^{k} X^{(j,k)}~\stackrel{d}{=}~X^{(1,1)}
\end{equation}
for all $k\in\N$, and given $\sum_{j=1}^{k}X^{(j,k)}=x$, the vector $(X^{(1,k)},\dots,X^{(k,k)})$ has a multinomial distribution with parameters $x$ and $q_{1}(k),\dots,q_{k}(k)\in[0,1]$. Replacing \eqref{eq:multinomial sharing} with
$$ X^{(j,k)}~\stackrel{d}{=}~X^{(1,1)}\quad\text{for all}\ 1\le j\le k< \infty, $$
we arrive at a case in which the number of offspring shared by a parasite into a daughter cell is always the same and thus independent of the number of such cells produced. This implies that along an infinite cell line the number of hosted parasites forms an ordinary GWP. As yet another specialization, one may consider the situation when
$$ X^{(j,k)} = 0\quad\text{a.s. for all } 2\le j\le k<\infty, $$
while $X^{(1,1)}, X^{(1,2)}, \dots$ are independent and positive-valued. Hence, starting with a single cell, parasites are only located in the leftmost cell $1_{n}=1...1$ $(n$ times$)$ in each generation $n$ and $\calZproc$ forms a \emph{GWP in random environment (GWPRE)}, with the number of daughter cells of each $1_{n}$ forming the (iid) random environment and hence determining the offspring distribution in each generation (see \cite{AthreyaKarlin:71a, Smith+Wilkinson:69} for the definition of a GWPRE).

\vspace{.1cm}
Regarding our model assumptions, let us note that it is not for pure mathematical generality when allowing cell division into more than two daughter cells. In the standard case, one of the two daughter cells after cell division may be viewed as the original mother cell which accumulates age-related damage throughout its replication phases and eventually loses the ability for cellular mitosis so that cell death occurs. This phenomenon, called cellular senescence, has been discovered recently even for several single-celled organisms (monads), see Stephens \cite{Stephens:05}. Genealogical aspects may be studied with the help of Galton-Watson trees when counting all cells stemming from a single cell during its lifetime and interpreting them as the succeeding generation.
Since the infection level of the mother cell changes during its lifetime, different numbers of parasites in the daughter cells are to be expected, thus justifying the assumption of asymmetric sharing of parasites. As another reason for this assumption is the fundamental biological mechanism to generate cell diversity, see Jan and Jan \cite{JanJan:98} and Hawkins and Garriga \cite{HawkinsGarriga:98}. For example, a stem cell uses asymmetric sharing to give rise to a copy of itself and a second daughter cell which is coded to differentiate into cells with a particular functionality in the organism.

\vspace{.1cm}
The purpose of this article is to establish equivalent conditions for almost sure extinction of parasites. The limiting behavior of the BwBP will be studied in a companion paper \cite{AlsGroettrup:15b}.

\section{The associated branching process in random environment}\label{Sec.ABPRE}

The first step towards an analysis of the BwBP is to identify a certain \emph{infinite random cell line} through the cell tree $\T$ and to study its properties. This approach was first used by Bansaye in \cite{Bansaye:08} who simply picked a random path in the infinite binary Ulam-Harris tree representing the cell population. Since the cell tree is here a general GWT and thus random, we must proceed in a different manner already introduced in \cite{AlsGroettrup:13}. In fact, we pick the random path $(V_{n})_{n\ge 0}$ in $\V$ according to a certain size-biased distribution instead of uniformly. 

\vspace{.1cm}
To give details, let $(C_{n}, T_{n})_{n\ge 0}$ be a sequence of iid random vectors independent of $\Tfam$ and $\Xfam$. The law of $T_{n}$ equals the size-biasing of the law of $N$, i.e.
\begin{equation*}
\P(T_{n}=k)=\frac{kp_k}{\nu}
\end{equation*}
for each $n\in\N_{0}$ and $k\in\N$, and
\begin{equation*}
 \P(C_{n}=l|T_{n}=k)=\frac{1}{k}
\end{equation*}
for $1\le l\le k$, which means that $C_{n}$ has a uniform distribution on $\{1,\dots,k\}$ given $T_{n}=k$. Now, $(V_{n})_{n\ge 0}$ is recursively defined by $V_{0}=\varnothing$ and
\begin{equation*}
 V_{n}:= V_{n-1} C_{n-1}
\end{equation*}
for $n\ge 1$. Then
\begin{equation*}
 \varnothing=:V_{0}\to V_{1}\to V_{2}\to\dots\to V_{n}\to\dots
\end{equation*}
provides us with a random cell line in $\V$ (not picked uniformly) as depicted in Figure \ref{Fig.SpinalCellTree}.
\begin{figure}[t]
\centering
\includegraphics[width=10cm]{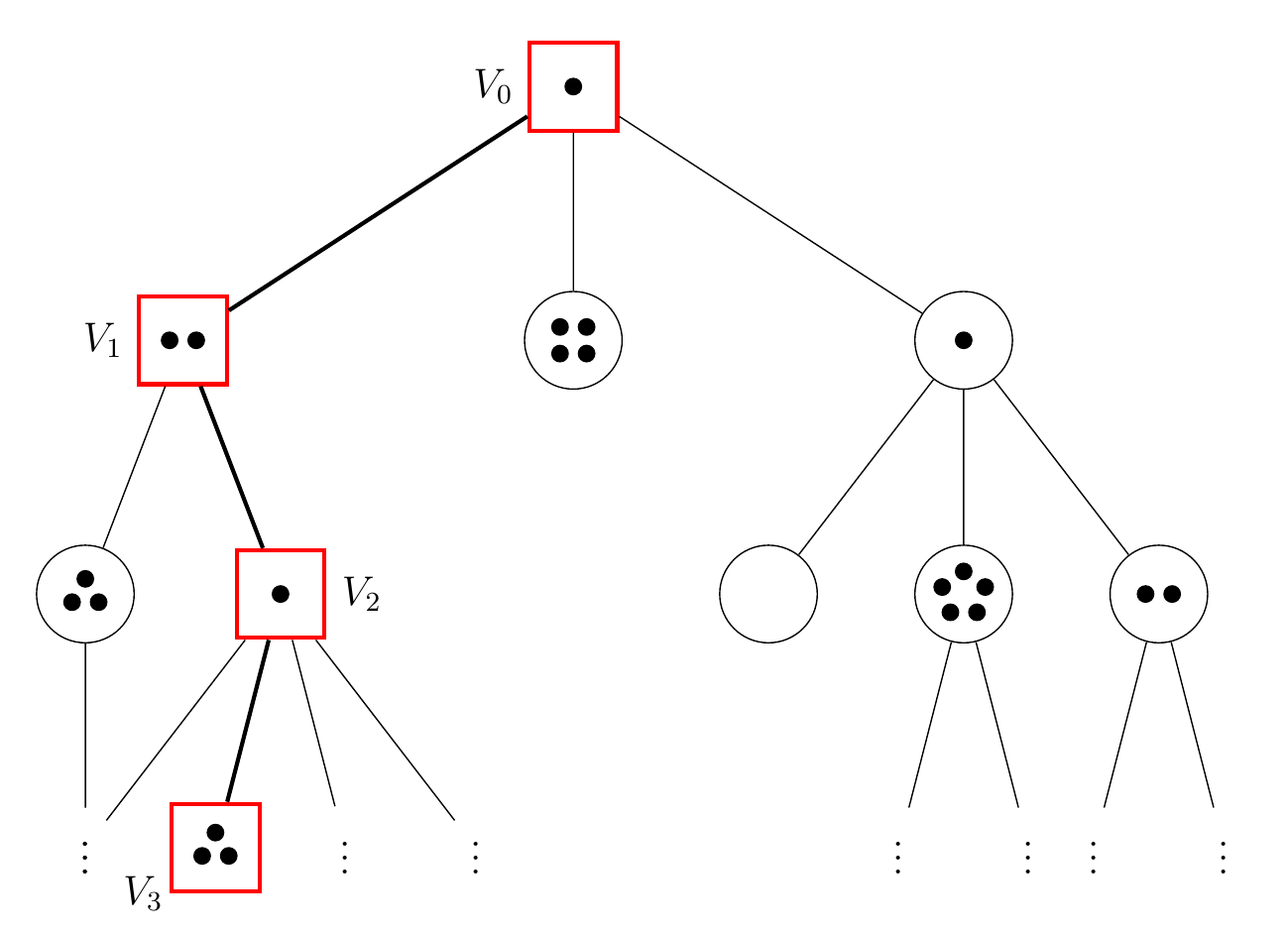}
\caption{Typical realization of a spine in the size-biased cell tree. A spinal cell is shown as a red square $\color{red}\square$, all other cells as a circle $\bigcirc$.}
\label{Fig.SpinalCellTree}
\end{figure}
The resulting path can be viewed as a so-called \emph{spine} in a size-biased tree, each spine cell representing a ``typical'' cell of its generation in the ordinary cell tree. The concept of size-biasing in the branching context goes back to Lyons et al. in \cite{LyPePe:95}, who used it to give alternative proofs of classical limit theorems for GWP's, and we refer to their work for a detailed construction of a spinal GWT.

\vspace{.1cm}
Concentrating on the number of parasites along $(V_{n})_{n\ge 0}$, we have $Z_{V_{0}}=Z_\varnothing$ and, recursively,
\begin{equation*}
Z_{V_{n+1}}\ =\ \sum_{k=1}^\infty\sum_{j=1}^{k}\1_{\{C_{n}=j,\,T_{n}=k\}}\sum_{i=1}^{Z_{V_{n}}}X_{i,V_{n}}^{(j,k)}\ =\ \sum_{i=1}^{Z_{V_{n}}}X_{i,V_{n}}^{(C_{n},T_{n})}.
\end{equation*}
for $n\ge 0$. Hence, given $(C_{n},T_{n})$, all parasites in generation $n$ produce offspring independently with the same distribution $\calL(X^{(C_{n},T_{n})})$. Since the $(C_{n}, T_{n})$ are iid and independent of $\Xfam$, the process of parasites along the spine $(Z_{V_{n}})_{n\ge 0}$ forms a \emph{branching process in iid random environment (BPRE)} as defined e.g. in \cite{AthreyaKarlin:71a, Smith+Wilkinson:69}. We summarize this observation in the following lemma (see also \cite[Subsection 2.1]{AlsGroettrup:13}).

\begin{Lemma}\label{lem:ABPRE}
Let $\assBPRE$ be a BPRE with $Z_\varnothing$ ancestors and iid environmental sequence $\Lambda:=(\Lambda_{n})_{n\ge 0}$ taking values in $\{\calL(X^{(j,k)})|1\le j\le k<\infty\}$ and such that
\begin{equation*}
\P\left(\Lambda_{0}=\calL(X^{(j,k)})\right)=\frac{p_{k}}{\nu}
\end{equation*}
for all $1\le j\le k<\infty$. Then $(Z_{V_{n}})_{n\ge 0}$ and $\assBPRE$ are equal in law.
\end{Lemma}

\begin{proof}
It suffices to point out that, by assumption, $Z^\prime_{0}=Z_\varnothing$ a.s. and
\begin{equation*}
\P\left(\Lambda_{0}=\calL(X^{(j,k)})\right)\ =\ \frac{p_{k}}{\nu}\ =\ \frac{1}{k}\cdot\frac{kp_{k}}{\nu}\ =\ \P\left(C_{0}=j,T_{0}=k\right)
\end{equation*}
for all $1\le j\le k<\infty$, i.e., both processes pick their reproduction law in each generation by the same random mechanism.\qed
\end{proof}

The BPRE $\assBPRE$ with environmental sequence $\Lambda$ is called hereafter the \emph{associated branching process in random environment (ABPRE)}. It is one of the major tools used in the study of the BwBP, and the following proposition provides a key relation between this process and its ABPRE.

\begin{Prop}\label{Prop.ABPRE.T}
For all $n,k,z\in\N_{0}$,
\begin{equation}\label{Eq.BPRE.EG}
\P_{z}\left(Z_{n}'=k\right)=\nu^{-n}\,\E_{z}\calT_{n,k}\quad\text{and}\quad\P_{z}\left(Z_{n}'>0\right)=\nu^{-n}\,\E_{z}\calT_{n}^{*}.
\end{equation}
\end{Prop}

\begin{proof}
By definition of the involved random variables, we find that
\begin{align*}
\P_{z}(N_{\sfv|0}=k_{0},\dots, N_{\sfv|n-1}=k_{n-1})\ &=\ \prod_{i=0}^{n-1}p_{k_i}\\
&=\ \nu^{n}\,\P_{z}(V_{n}=\sfv,\ T_{0}=k_{0},\dots, T_{n-1}=k_{n-1})
\end{align*}
and
\begin{align*}
\P_{z}(Z_{\sfv}=k|N_{\sfv|0}=k_{0},&\dots, N_{\sfv|n-1}=k_{n-1})\\
&=\ \P_{z}(Z_{\sfv}=k|V_{n}=\sfv,\ T_{0}=k_{0},\dots, T_{n-1}=k_{n-1})
\end{align*}
for all $n,k\in\N$, vertices $\sfv=\sfv_{1}\dots\sfv_{n}$ and $k_{0}\ge\sfv_{1},\dots, k_{n-1}\ge\sfv_{n}$, where $\sfv|0:=\varnothing$ and $\sfv|j:=\sfv_{1}\ldots\sfv_{j}$ for $1\le j\le n=|\sfv|$. Hence it follows by summation over the $k_{i}$ that  
$$  \P_{z}(Z_{\sfv}=k)\ =\ \nu^{n}\,\P_{z}(Z_{\sfv}=k, V_{n}=\sfv)$$
for all $\sfv\in\V$ with $|\sfv|=n$, which in turn leads to
\begin{align*}
\nu^{n}\,\P_{z}(Z_{V_{n}}=k)\ &=\ \sum_{|\sfv|=n}\nu^{n}\,\P_{z}(Z_{\sfv}=k, V_{n}=\sfv)\\
&=\ \sum_{|\sfv|=n}\P_{z}(Z_{\sfv}=k)\ =\ \sum_{|\sfv|=n}\E_{z}\1_{\{Z_{\sfv}=k\}}\ =\ \E_{z}\calT_{n,k}
\end{align*}
and proves the first assertion in view of Lemma \ref{lem:ABPRE}. The second equation then follows by summation over all $k>0$.
\end{proof}

We finish this section with a quick review of some relevant facts about the BPRE, relevant references being \cite{AfaGKV:05b, AfaGKV:05,Agresti:75, AthreyaKarlin:71b, AthreyaKarlin:71a,Bansaye:09b,Coffey+Tanny:84,GeKeVa:03,Smith+Wilkinson:69, Tanny:77, Tanny:77/78, Tanny:88}. For $n\in\N$ and $s\in [0,1]$,
\begin{equation*}
\E(s^{Z_{n}'}|\Lambda)\ =\ g_{\Lambda_{0}}\circ...\circ g_{\Lambda_{n-1}}(s)
\end{equation*}
is the quenched generating function of $Z_{n}'$ with iid $g_{\Lambda_{n}}$ and $g_{\lambda}$ defined by
\begin{equation*}
g_{\lambda}(s)\ :=\ \E(s^{Z^\prime_{1}}|\Lambda_{0}\ =\ \lambda)=\sum_{n\ge 0}\lambda_{n}s^{n}
\end{equation*}
for any distribution $\lambda=(\lambda_{n})_{n\ge 0}$ on $\N_{0}$. Moreover,
\begin{equation}\label{Eq.EW.BPRE}
\E g_{\Lambda_{0}}'(1)\ =\ \E Z_{1}'\ =\ \sum_{1\le j\le k}\frac{p_{k}}{\nu}\ \E X^{(j,k)}\ =\ \frac{\E\calZ_{1}}{\nu}\ =\ \frac{\gamma}{\nu}\ <\ \infty,
\end{equation}
where $\gamma = \E\calZ_{1}$. 
It is also well-known that $\assBPRE$ survives with positive probability iff
\begin{align}\label{Eq.BPRE.Aussterben}
\E\log g_{\Lambda_{0}}'(1)> 0\quad\text{and}\quad\E\log^{-}(1-g_{\Lambda_0}(0))<\infty,
\end{align}
see e.g. \cite{AthreyaKarlin:71a, Smith+Wilkinson:69}, and recall that $\gamma<\infty$ is assumed by \eqref{As.Gamma}. Furthermore, by \eqref{As.neq1}, there exists $1\le j\le k<\infty$ such that $p_{k}>0$ and $\P(X^{(j,k)}\neq1)>0$, which ensures that $\Lambda_{0}\ne\delta_{1}$ with positive probability. The ABPRE is called \emph{supercritical, critical} or \emph{subcritical} if $\E\log g_{\Lambda_{0}}'(1)>0$, $=0$ or $<0$, respectively. The subcritical case further divides into the three subregimes when $\E g_{\Lambda_{0}}'(1)\log g_{\Lambda_{0}}'(1)<0, =0$, or $>0$, respectively, called \emph{strongly, intermediate} and \emph{weakly subcritical case}. The quite different behavior of the process in each of the three subregimes is shown by the limit results derived in \cite{GeKeVa:03}.

\vspace{.1cm}
Focussing on the subcritical case hereafter, thus $\E\log g_{\Lambda_{0}}'(1)<0$, we point out the following useful facts. If $\E g_{\Lambda_{0}}'(1)\log g_{\Lambda_{0}}'(1)\le  0$, then the convexity of $\theta\mapsto\E g_{\Lambda_{0}}'(1)^{\theta}$ implies that
\begin{equation*}
\E g_{\Lambda_{0}}'(1)\ =\ \inf_{0\le \theta\le 1}\E g_{\Lambda_{0}}'(1)^{\theta}.
\end{equation*}
Under the assumptions
\begin{equation*}\tag{AsGe}\label{As.GeigerSubcriticalLimit}
 \P(Z^{\prime}_{1}\le  C)=1\quad \text{and}\quad \P\left(0<g_{\Lambda_{0}}'(1)<\eps\right)=0
\end{equation*}
for suitable constants $C>0$ and $\eps>0$, Geiger et al. \cite[Thms. 1.1--1.3]{GeKeVa:03} showed that
\begin{equation}\label{Eq.GeigerSurvivalEstimate}
\P(Z_{n}^\prime>0)\ \simeq\ cn^{-\kappa}\left(\inf_{0\le \theta\le  1}\E g_{\Lambda_{0}}'(1)^{\theta}\right)^{n}\quad\text{as }n\to\infty
\end{equation}
for some $c\in (0,\infty)$, where
\begin{equation*}
\kappa\ =\ 
\begin{cases}
0 & \text{if $\E g_{\Lambda_{0}}'(1)\log g_{\Lambda_{0}}'(1)<0$}\quad\text{(strongly subcritical case)},\\[1ex]
\frac{1}{2}& \text{if $\E g_{\Lambda_{0}}'(1)\log g_{\Lambda_{0}}'(1)=0$}\quad\text{(intermediate subcritical case)},\\[1ex]
\frac{3}{2}& \text{if $\E g_{\Lambda_{0}}'(1)\log g_{\Lambda_{0}}'(1)>0$}\quad\text{(weakly subcritical case)}.
\end{cases}
\end{equation*}
The condition \eqref{As.GeigerSubcriticalLimit} can be strongly relaxed for the asymptotic relation of the survival probability in \eqref{Eq.GeigerSurvivalEstimate} to hold, see for example \cite{GeKeVa:03, Vatutin:04}, but is enough for our purposes.

A combination of \eqref{Eq.BPRE.EG} and \eqref{Eq.GeigerSurvivalEstimate} provides us with the asymptotic relation
\begin{equation}\label{Eq.TnAndGeiger}
\E\calT_{n}^{*}\ \simeq\ cn^{-\kappa}\nu^{n}\left(\inf_{0\le \theta\le 1}\E g_{\Lambda_{0}}'(1)^{\theta}\right)^{n}\quad\text{as }n\to\infty,
\end{equation}
in particular (with \eqref{As.GeigerSubcriticalLimit} still in force)
\begin{equation}\label{Eq.TnAndGeiger2}
\inf_{0\le\theta\le  1}\E g_{\Lambda_{0}}'(1)^{\theta}\ \le\ \frac{1}{\nu}\quad\text{if}\quad\sup_{n\ge 1}\E\calT_{n}^{*}<\infty.
\end{equation}

\section{The extinction problem}\label{sec:extinction}

The trivial fact that non-contaminated cells are unable to produce infected daughter cells shows that the process
$$ \calbfT_{n}^{*}\ :=\ \left(\calT_{n,1}, \calT_{n,2}, \dots\right)$$
also forms a MTBP with type space $\{1,2,3,\dots\}$. As a consequence, the following \emph{extinction-explosion principle} is easily obtained for our model. 

\begin{Theorem}[Extinction-explosion principle]\label{Ext-Exp-Princ}
The parasite population of a BwBP either dies out or explodes, i.e. for all $\vec{s}\in\bfN$
\begin{equation*}
\P_{\vec{s}}(\calZ_{n}\to 0)\ +\ \P_{\vec{s}}(\calZ_{n}\to\infty)\ =\ 1. 
\end{equation*}
\end{Theorem}

\begin{proof}
Let $\vec{s}\in\bfN$ with $s_{i}>0$ for at least one $i\ge 1$. We prove that $\vec{s}$ is a transient state for the Markov chain $\bfTprocStar$. Consider three cases: If $\P(N=0)>1$, then 
$$ \P\left(\calbfT_{n}^{*}\neq\vec{s}\ \text{for all}\ n\ge1|\calbfT_{0}^{*}=\vec{s}\right)~\ge~\P(N=0)^{\sum_{i}s_{i}}\ >\ 0. $$
Otherwise, if $\P(N=0)=0$ but $\P(\calZ_{1}=0)>0$, then there exists a $k\in\N$ such that $p_{k}\,\P(\sum_{j=1}^{k}X^{(j,k)}=0)>0$ and thus
$$ \P\left(\calbfT_{n}^{*}\neq\vec{s}\ \text{for all}\ n\ge1|\calbfT_{0}^{*}=\vec{s}\right)~\ge~p_{k}^{\sum_{i}s_{i}}\P\left(\sum_{j=1}^{k}X^{(j,k)}=0\right)^{\sum_{i}i s_{i}}\ >\ 0.$$
Finally, consider the case when $\P(N=0)=0$ and $\P(\calZ_{1}=0)=0$. Recalling \eqref{As.Constant1}, we then have $\P(\calZ_{1}>1)>0$ and thus $p_{k}\,\P(\sum_{j=1}^{k}X^{(j,k)}>1)>0$ for some $k\in\N$. But this implies
$$ \P\left(\calbfT_{n}^{*}\neq\vec{s}\ \text{for all}\ n\ge 1|\calbfT_{0}^{*}=\vec{s}\right)\ \ge\ p_{k}\,\P\left(\sum_{j=1}^{k}X^{(j,k)}>1\right)\ >\ 0, $$
that is, $\vec{s}$ is a transient state of the Markov chain $\bfTprocStar$. Consequently, we infer that almost surely
$$\sum_{i\ge1}\calT_{n,i}\to 0,\quad\sum_{i\ge1}\calT_{n,i}\to\infty\quad\text{or, a fortiori,}\quad \calT_{n,i}\to\infty\ \text{for at least one } i\ge 1 $$
as $n\to\infty$. In any case, the parasite population a.s. dies out or tends to infinity.\qed
\end{proof}

We denote by
\begin{equation*}
\Ext:=\{\calZ_{n}\to 0\}\quad\text{and}\quad\Surv:=\Ext^c=\{\calZ_{n}\to\infty\}
\end{equation*}
the \emph{event of extinction} and \emph{of survival of parasites}, respectively, and put
\begin{equation*}
\P^*_{\vec{s}}:=\P_{\vec{s}}(\cdot|\Surv)\quad\text{and}\quad\E^*_{\vec{s}}:=\E_{\vec{s}}(\cdot|\Surv),
\end{equation*}
for $\vec{s}\in \bfN$. Also, let $\P^*_{z}$ and $\E^*_{z}$ for $z\in\N$ have the obvious meaning.

\vspace{.2cm}
Having shown that $\calZproc$ satsifies the extinction-explosion dichotomy, we will now turn to the process of contaminated cells and prove that, ruling out a degenerate case, survival of parasites always goes along with the number of contaminated cells tending to infinity. In other words, $\calT_{n}^{*}$ tends to infinity as $n\to\infty$ if this holds true for $\calZ_{n}$. The degenerate case occurs if all parasites sitting in the same cell send their offspring to the same daughter cell which formally means that, for each $k\ge 1$,
$$ X^{(\bullet,k)}\ =\ (0,\ldots,0,X^{(j(k),k)},0,\ldots,0) $$
for some unique $j(k)\in\{1,...,k\}$ or, equivalently,
$$ \P_{2}(\calT_{1}^{*}\ge 2)\ =\ 0. $$
Note that it is enough to consider a single root cell due to the branching property.

\begin{Theorem}\label{Th.ExplosionInfZellen} 
Let $\P(\Surv)>0$.
\begin{description}[(b)]
\item[(a)] If $\P_{2}(\calT_{1}^{*}\ge 2)>0$, then $\P^*_{z}(\calT_{n}^{*}\to\infty)=1$ and thus $\Ext=\{\sup_{n\ge 0}\calT_{n}^{*}<\infty\}$ $\P_{z}$-a.s. for all $z\in\N$.
\item[(b)] If $\P_{2}(\calT_{1}^{*}\ge 2)=0$, then $\P^*_{z}(\calT_{n}^{*}=1\text{ f.a. }n\ge 0)=1$ for all $z\in\N$.
\end{description}
\end{Theorem}

\begin{proof}
Fix any $z\in\N$ and consider first the easier case $(b)$. Note that $\P_{2}(\calT_{1}^{*}\ge 2)=0$ implies $\P_{z}(\calT_{n}^{*}\le 1\ \forall\ n\ge 0)=1$. But since $\Surv = \{\calT_{n}^{*}\ge 1\text{ f.a. }n\ge 0\}$ $\P_{z}$-a.s., $(b)$ is proved.

\vspace{.1cm}
For the proof of $(a)$, we use the Markov chain $\bfTprocStar$ to show that $\calTprocStar$ visits each $t\ge 1$ only finitely often, that is
\begin{equation}\label{extinction explosion for T_{n}^*}
 \P_{z}(1\le  \calT_{n}^{*}\le  t\ \text{infinitely often})=0
\end{equation}
for all $t\ge 1$, hence $\P_{z}(\lim_{n\to\infty}\calT_{n}^{*}=0\text{ or }\infty)=1$. But since $\Ext=\{\calT_{n}^{*}\to 0\}$ $\P_{z}$-a.s., $(a)$ follows.

\vspace{.1cm}
Left with the proof of \eqref{extinction explosion for T_{n}^*}, we define (with $\bfN$ given by \eqref{def of bfN})
\begin{equation*}
 A_{k}\ :=\ \left\{ \vec{s}\in\bfN \ \Big|\ \sum_{i\ge 1}s_{i}=k,\ s_{1}\ne k\right\}\ \subseteq\ \bfN
\end{equation*}
for $k\ge 1$ and observe that, for $n\ge 0$,
\begin{equation*}
\{\calT_{n}^{*}=k\}\ =\ \{\calbfT_{n}^{*}\in A_{k}\}\cup\{\calbfT_{n}^{*}\ =\ (k,0,0,\dots)\}\quad\P_{z}\text{-a.s.}
\end{equation*}
Since $(k,0,0,\dots)\in \bfN$ is a transient state by Theorem \ref{Ext-Exp-Princ}, we obtain
\begin{equation*}
\P_{z}(\calT_{n}^{*}=k\ \text{infinitely often})\ =\ \P_{z}(\calbfT_{n}^{*}\in A_{k}\ \text{infinitely often}).
\end{equation*}
Therefore, it remains to prove that the Markov chain $\bfTprocStar$ visits the set $A_{k}$ only finitely often with probability 1. For $\vec{s}\in A_{k}$, i.e. $\sum_{i}s_{i}=k$ and $s_{j}>1$ for some $j>1$, we infer with the help of the branching property 
\begin{align*}
\P_{\vec{s}}&\left(\calbfT_{n}^{*}\notin A_{k}\ \text{for all}\ n\ge 1\right)\\
&\ge\ \P_{\vec{s}}\left(\calT_{n}^{*}>k\ \text{for all}\ n\ge 1\right) \\
&\ge\ \P_{j}\left(\calT_{1}^{*}\ge 2\right)\,\P(\Surv)^{2}\prod_{i:s_{i}>0}\big(\P_{i}\left(\calT_{1}^*\ge 1\right)\P(\Surv)\big)^{s_{i}}\\
&\ge\ \P_{2}\left(\calT_{1}^{*}\ge 2\right)\,\P\left(\calT_{1}^*\ge 1\right)^{k}\P(\Surv)^{k+2}
\end{align*}
and positivity of the last expression is guaranteed by our assumptions. Notice also that this expression is independent of the choice of $\vec{s}$. Next, let $\tau_{0}:=0$ and
\begin{equation*}
\tau_{n+1}:=\inf\left\{m>\tau_{n}|\calbfT^*_m\in A_{k}\right\}
\end{equation*}
for $n\ge 0$ be the successive return times of $\bfTprocStar$ to the set $A_{k}$, where as usual $\inf\emptyset:=\infty$ and $\tau_{n+1}=\infty$ if $\tau_{n}=\infty$. Then the previous inequality and the strong Markov property of $\bfTprocStar$ imply the existence of a constant $c<1$ such that \emph{for all} $\vec{s}\in A_{k}$ and $n\ge 0$
\begin{align*}
 \P_{z}\left(\tau_{n+1}-\tau_{n}<\infty|\calbfT_{\tau_{n}}^{*}=\vec{s}, \tau_{n}<\infty\right)
\ &=\ \P_{\vec{s}}\left(\tau_{1}<\infty\right)\ \le\ c\ <\ 1.
\end{align*}
Therefore, we infer upon iteration that
\begin{align*} 
\P_{z}&(\tau_{n}<\infty)\\
&=\ \sum_{\vec{s}\in A_{k}}\P_{z}(\calbfT_{\tau_{n-1}}^{*}=\vec{s}, \tau_{n}-\tau_{n-1}<\infty, \tau_{n-1}<\infty)\\
&=\sum_{\vec{s}\in A_{k}}\P_{z}(\tau_{n}-\tau_{n-1}<\infty|\calbfT_{\tau_{n-1}}^{*}=\vec{s}, \tau_{n-1}<\infty)\,\P_{z}(\calbfT_{\tau_{n-1}}^{*}=\vec{s}, \tau_{n-1}<\infty)\\
&\le\ c\,\P_{z}(\tau_{n-1}<\infty)\ \le\ \ldots\ \le\ c^{n-1}\P_{z}(\tau_{1}<\infty)\ \le\ c^{n-1}
\end{align*}
for each $n\ge 1$ and thereupon
\begin{align*}
\P_{z}\left(\calbfT_{n}^{*}\in A_{k}\ \text{infinitely often}\right)\ &=\ \P_{z}\left(\tau_{n}<\infty\text{ for all }n\ge 1\right)\\
&=\ \P_{z}\left(\bigcap_{n\ge 1}\{\tau_{n}<\infty\}\right)\\
&=\ \lim_{n\to\infty}\P_{z}(\tau_{n}<\infty)\ \le\ \lim_{n\to\infty}c^{n-1}\ =\ 0
\end{align*}
which completes the proof.\qed
\end{proof}

So we have verified an extinction-explosion dichotomy for both, the process of contaminated cells and of parasites, and will now proceed with the main result in this paper which provides equivalent conditions for almost sure extinction of parasites. The proof will make use of the ABPRE introduced in Section \ref{Sec.ABPRE}.

\begin{Theorem}\label{Th.FastSicheresAussterben}\
\begin{description}[(b)]
\item[(a)] If $\P_{2}(\calT_{1}^{*}\ge 2)=0$, then $\P(\Ext)=1$ if, and only if,
\begin{equation*}
\E\log \E\left(\calZ_{1}|N_{\varnothing}\right)\le 0\quad \text{or}\quad \E\log^{-}\P\left(\calZ_{1}>0|N_{\varnothing}\right)=\infty.%
\end{equation*}
\item[(b)] If $\P_{2}(\calT_{1}^{*}\ge 2)>0$, then the following statements are equivalent:
		\begin{enumerate}
 		\item[(i)] $\P(\Ext)=1$.\vspace{.08cm}
		\item[(ii)] $\E\calT_{n}^{*}\le 1$ for all $n\in\N_0$.\vspace{.08cm}
		\item[(iii)] $\sup_{n\in\N_0}\E\calT_{n}^{*}<\infty$.\vspace{.08cm}
		\item[(iv)] $\nu\le 1$, or
			\begin{equation*}
			\nu>1,\quad\E\log g_{\Lambda_{0}}'(1)<0\quad\text{and}\quad\inf_{0\le \theta
			\le 1}\E g_{\Lambda_{0}}'(1)^{\theta}\le  \frac{1}{\nu}.
			\end{equation*}
		\end{enumerate}
\end{description}
\end{Theorem}

\begin{proof}
(a) If $\P_{2}(\calT_{n}^{*}\ge 2)=0$, then for all $k\ge 1$ with $p_{k}>0$ there exists at most one $1\le j\le k$ such that $\P(X^{(j,k)}>0)>0$, see before Theorem \ref{Th.ExplosionInfZellen}. As a consequence, $\calZproc$ is a branching process in random environment, the latter given by the iid numbers of daughter cells produced by the unique cells which contain the parasites (their common law being the law of $N_{\varnothing}$). Hence, $\calZproc$ dies out almost surely if, and only if, $\E(\log \E(\calZ_{1}|N_{\varnothing}))\le 0$ or $\E\log^{-}\P(\calZ_{1}>0|N_{\varnothing})=\infty$ (see e.g. \cite{Smith+Wilkinson:69}).

\vspace{.2cm}
(b) Suppose now $\P_{2}(\calT_{1}^{*}\ge 2)>0$.

\vspace{.1cm}
``$(i)\Rightarrow(ii)$'' (by contraposition) Fix $m\in\N$ such that $\E\left(\calT_{m}^{*}\right)>1$ and consider a supercritical GWP $(S_{n})_{n\ge 0}$ with $S_0=1$ and offspring distribution
\begin{equation*}
\P(S_{1}=k)\ =\ \P(\calT_{m}^{*}=k),\quad k\in\N_{0}.
\end{equation*}
Obviously,
\begin{equation*}
 	\P(S_{n}>k)\ \le\ \P(\calT_{nm}^{*}>k)
\end{equation*}
for all $k,n\in\N_0$, hence 
\begin{equation*}
 \lim_{n\to\infty}\P(\calT_{nm}^{*}>0)\ \ge\ \lim_{n\to\infty}\P(S_{n}>0)>0,
\end{equation*}
i.e. parasites survive with positive probability.

\vspace{.1cm}
``$(ii)\Rightarrow(iii)$'' is trivial.

\vspace{.1cm}
``$(iii)\Rightarrow(i)$''
Recall that $\liminf_{n\to\infty}\calT_{n}^{*}=\infty$ a.s. on $\Surv$ by Theorem \ref{Th.ExplosionInfZellen}. On the other hand, $\sup_{n\ge 0}\E\calT_{n}^{*}<\infty$ implies
\begin{equation*}
 	\infty\ >\ \liminf_{n\to\infty} \E\calT_{n}^{*}\ \ge\ \E\left(\liminf_{n\to\infty}\calT_{n}^{*}\right)
\end{equation*}
by Fatou's lemma so that $\P(\Surv)=0$.

\vspace{.2cm}
``$(iv)\Rightarrow(i),(ii)$''
If $\nu\le 1$, then $\E\calT_{n}^{*}\le \E\calT_{n}=\nu^{n}\le  1$ for all $n\in\N$. So let us consider the situation when
\begin{equation*}
\nu>1,\quad\E\log g_{\Lambda_{0}}'(1)<0\quad\text{and}\quad\inf_{0\le \theta\le 1}\E g_{\Lambda_{0}}'(1)^{\theta}\le \frac{1}{\nu} 
\end{equation*}
is valid. It is here where the ABPRE $(Z_{n}')_{n\ge 0}$ comes into play. By \eqref{Eq.BPRE.EG},
\begin{equation*}
 	\E\calT_{n}^{*}\ =\ \nu^{n}\,\P(Z^\prime_{n}>0)
\end{equation*}
for all $n\in\N$. We distinguish three cases:

\vspace{.2cm}
\textsc{Case A}. $\E g_{\Lambda_{0}}'(1)\log g_{\Lambda_{0}}'(1)\le 0$.\vspace{.1cm}

By \eqref{Eq.EW.BPRE} and what has been pointed out in the short review of BPRE's at the end of the previous section, we then infer
\begin{equation*}
 	\frac{\gamma}{\nu}\ =\ \E g_{\Lambda_{0}}'(1)\ =\ \inf_{0\le \theta\le 1}\E 
	g_{\Lambda_{0}}'(1)^{\theta}\le\frac{1}{\nu}
\end{equation*}
and thus $\gamma\le 1$, which in turn entails
\begin{equation*}
	\E\calT_{n}^{*}\ \le\ \E\calZ_{n}\ =\ \gamma^{n}\ \le\  1
\end{equation*}
for all $n\in\N$ as claimed.

\vspace{.2cm} 
\textsc{Case B}. $\E g_{\Lambda_{1}}'(1)\log g_{\Lambda_{1}}'(1)>0$ and \eqref{As.GeigerSubcriticalLimit}. 

\vspace{.1cm}
Then, by \eqref{Eq.GeigerSurvivalEstimate},
\begin{equation*}
 	\P(Z_{n}^\prime>0)\ \simeq\ c n^{-3/2} \left(\inf_{0\le \theta\le  1}
	\E g_{\Lambda_{1}}'(1)^{\theta}\right)^{n}\quad\text{as } n\to\infty
\end{equation*}
holds true for a suitable constant $c\in (0,\infty)$ whence, by Fatou's lemma,
\begin{equation*}
	0\ =\ \lim_{n\to\infty} \nu^{n}\,\P(Z^\prime_{n}>0)\ =\ \liminf_{n\to\infty} \E\calT_{n}^{*}\ \ge\ \E\left(\liminf_{n\to\infty}\calT_{n}^{*}\right).
\end{equation*}
Consequently, $\P(\Surv)=0$ since $\inf_{n\ge 0}\calT_{n}^{*}\ge 1$ a.s. on $\Surv$.

\vspace{.2cm} 
\textsc{Case C}. $\E g_{\Lambda_{1}}'(1)\log g_{\Lambda_{1}}'(1)>0$.

\vspace{.1cm}
Using contraposition, suppose $\sup_{n\in\N}\E\calT_{n}^{*}>1$ and fix an arbitrary vector $\alpha=(\alpha^{(j,k)})_{1\le  j\le k<\infty}$ of distributions on $\N_{0}$ satisfying
\begin{equation*}
 \alpha_{x}^{(j,k)}\ \le \ \P\left(X^{(j,k)}=x\right)
\end{equation*}
for $x\ge 1$ and $j,k$ as stated, hence
\begin{equation*}
 \alpha_{0}^{(j,k)}\ \ge \ \P\left(X^{(j,k)}=0\right)\quad\text{and}\quad\sum_{x\ge  n}\alpha_{x}^{(j,k)}\ \le \ \P\left(X^{(j,k)}\ge  n\right) 
\end{equation*}
for each $n\ge  0$. Possibly after enlarging the underlying probability space, we can then construct a BwBP $(N_{\sfv}, Z_{\alpha,\sfv}, X^{(\bullet,k)}_{\alpha,i,\sfv})_{\sfv\in\V, i,k\in\N}$ coupled with and of the same kind as the original BwBP such that
\begin{align*}
\P\left(X^{(j,k)}_{\alpha,i,\sfv} = x \right)\ =\ \alpha_{x}^{(j,k)}\quad\text{and}\quad X^{(j,k)}_{\alpha, i,\sfv}\ \le \ X^{(j,k)}_{i,\sfv}\quad\text{a.s.}
\end{align*}
for all $1\le  j\le k<\infty$, $\sfv\in\V$, $i\ge 1$ and $x\ge  1$.
Then $Z_{\alpha,\sfv}\le  Z_{\sfv}$ a.s.\ for all $\sfv\in\V$ and since the choice of $\alpha$ has no affect on the cell splitting process, we have $\nu_{\alpha}=\nu>1$ and thus for $\theta\in[0,1]$
\begin{align}
\begin{split}\label{Eq.GestutzterProzess.G}
\E g_{\alpha,\Lambda_{0}}'(1)^{\theta}\ &=\ \E\left(\E(Z^\prime_{\alpha,1}|\Lambda_0)^{\theta}\right) \\
&=\ \sum_{1\le  j\le k<\infty}\frac{p_{k}}{\nu}\left(\E X^{(j,k)}_{\alpha}\right)^\theta\\
&\le\ \sum_{1\le  j\le k<\infty}\frac{p_{k}}{\nu}\mu_{j,k}^\theta\ \le\ \E g_{\Lambda_{0}}'(1)^{\theta}
\end{split}
\end{align}
where $\nu_{\alpha}$, $Z^\prime_{\alpha,n}$, $X^{(j,k)}_{\alpha}$ and $g_{\alpha,\Lambda_{0}}$ have the obvious meaning. Recalling $\mu_{j,k}=\E X^{(j,k)}$,  a similar calculation as in \eqref{Eq.GestutzterProzess.G} leads to
\begin{equation}\label{Eq.Gestutze.gf}
\E\log g_{\alpha,\Lambda_{0}}'(1)\ \le\ \E\log g_{\Lambda_{0}}'(1)\ <\ 0.
\end{equation}
We are now going to specify suitable $\alpha$ to complete our argument.

\vspace{.1cm}
For $M\in\N$ let $\alpha(M)=(\alpha^{(j,k)}(M))_{1\le  j\le k<\infty}$ be the vector defined by
\begin{equation*}
\alpha_{x}^{(j,k)}(M)\ :=\  
\begin{cases}
\P\left(X^{(j,k)} = x \right), &\text{if $1\le x\le M$},\\
\hfill 0, &\text{if $x>M$},
\end{cases}
\end{equation*}
if $\mu_{j,k}\ge 1/M$, and $\alpha_0^{(j,k)}=1$ if $\mu_{j,k}<1/M$. Then the BwBP with truncation $\alpha(M)$ satisfies the condition \eqref{As.GeigerSubcriticalLimit}, and we can fix $M\in\N$ such that $\sup_{n\in\N}\E\calT_{\alpha(M),n}^{*}>1$, because $\calT_{\alpha(M),n}^{*}\uparrow\calT_{n}^{*}$ as $M\to\infty$. Then, by what has already been proved under \textsc{Case B} in combination with \eqref{Eq.GestutzterProzess.G}, \eqref{Eq.Gestutze.gf} and $\nu_{\alpha(M)}>1$, we infer
\begin{equation}\label{Eq.AbschätzungInfGF>nu}
\inf_{0\le \theta\le 1}\E g_{\Lambda_{0}}'(1)^{\theta}\ \ge\ \inf_{0\le \theta\le 1}\E g_{\alpha(M),\Lambda_{0}}'(1)^{\theta}\ >\ \frac{1}{\nu}
\end{equation}
which contradicts $(iv)$.

\vspace{.2cm} 
``$(ii)\Rightarrow(iv)$''
Suppose that $\E\calT_{n}^{*}\le  1$ for all $n\in\N_0$ and further $\nu>1$ which, by 
\eqref{Eq.BPRE.EG}, entails $\lim_{n\to\infty}\P(Z^\prime_{n}>0)=0$, thus $\E\log g_{\Lambda_{0}}'(1)\le 0$ or $\E\log^{-}(1-g_{\Lambda_0}(0))=\infty$. We must show that $\E\log g_{\Lambda_{0}}'(1)<0$ and $\inf_{0\le \theta\le 1}\E g_{\Lambda_{0}}'(1)^{\theta}\le\nu^{-1}$. 

\vspace{.1cm}
Assuming $\E\log g_{\Lambda_{0}}'(1)<0$, the second condition follows from \eqref{Eq.TnAndGeiger2} if \eqref{As.GeigerSubcriticalLimit} is valid. Dropping the latter condition, suppose that $\inf_{0\le \theta\le 1}\E g_{\Lambda_{0}}'(1)^{\theta}>\nu^{-1}$. Then we can find a $M\ge 1$ and construct a suitable ``$\alpha(M)$-coupling'' as described above such that \eqref{Eq.AbschätzungInfGF>nu} holds. But since \eqref{As.GeigerSubcriticalLimit} is fulfilled for the truncated process, we arrive at the contradiction
\begin{equation*}
\sup_{n\in\N}\E\calT_{n}^{*}\ \ge\ \sup_{n\in\N}\E\calT_{\alpha(M),n}^{*}\ >\ 1
\end{equation*}
by referring to \eqref{Eq.Gestutze.gf} and by what has already been established for a BwBP with a subcritical ABPRE, i.e. $\E\log g_{\Lambda_{0}}'(1)<0$. Thus $\inf_{0\le \theta\le 1}\E g_{\Lambda_{0}}'(1)^{\theta}\le\nu^{-1}$ holds even if \eqref{As.GeigerSubcriticalLimit} fails. 

\vspace{.1cm}
It remains to rule out that $\E\log g_{\Lambda_{0}}'(1)\ge 0$. Assuming the latter, we find with the help of Jensen's inequality that
\begin{equation*}
\inf_{0\le \theta\le  1}\log \E g_{\Lambda_{0}}'(1)^{\theta}\ \ge\ \inf_{0\le \theta\le 1}\theta \,\E\log g_{\Lambda_{0}}'(1)\ \ge\  0
\end{equation*}
or, equivalently,
\begin{equation*}
	\inf_{0\le \theta\le 1}\E g_{\Lambda_{0}}'(1)^{\theta}\ \ge\ 1\ >\ \frac{1}{\nu}
\end{equation*}
(which implies $\inf_{0\le \theta\le 1}\E g_{\Lambda_{0}}'(1)^{\theta}=1$).
Use once more a suitable ``$\alpha$-coupling'' (not necessarily of the previous kind) and fix $\alpha$ in such a way that
\begin{equation*}
1\ =\ \inf_{0\le \theta\le 1}\E g_{\Lambda_{0}}'(1)^{\theta}\ >\ \inf_{0\le \theta\le 1}\E g_{\alpha,\Lambda_{0}}'(1)^{\theta}\ >\ \frac{1}{\nu},
\end{equation*}
which implies subcriticality of the ABPRE $(Z^\prime_{\alpha,n})_{n\ge 0}$ by taking the logarithm and using Jensen's inequality. As above, we thus arrive at the contradiction
\begin{equation*}
\sup_{n\in\N}\E\calT_{n}^{*}\ \ge\ \sup_{n\in\N}\E\calT^*_{\alpha,n}\ >\ 1
\end{equation*}
by using the already established results for a BwBP with subcritical ABPRE. This completes the proof of $(b)$.\qed
\end{proof}

\section*{Glossary}\label{sec:glossary}
\begin{description}[xxxxxxxxxx]
\item[$(p_{k})_{k\ge 0}$] offspring distribution of the cell population
\item[$\T$] cell tree in Ulam-Harris labeling
\item[$\T_{n}$] subpopulation at time (generation) $n$ $[=\{\sfv\in\T:|\sfv|=n\}]$
\item[$N_{\sfv}$] number of daughter cells of cell $\sfv\in\T$
\item[$\calT_{n}$] cell population size at time $n$ $[=\#\T_{n}=\sum_{|\sfv|=n}N_{\sfv}]$
\item[$\calT_{n,k}$] number of cells at time $n$ containing $k$ parasites
\item[$\calbfT_{n}$] the infinite vector $(\calT_{n,0},\calT_{n,1},...)$
\item[$\calbfT_{n}^{*}$] the infinite vector $(\calT_{n,1},\calT_{n,2},...)$
\item[$Z_{\sfv}$] number of parasites in cell $\sfv$
\item[$\calZ_{n}$] number of parasites at time $n$ $[=\sum_{\sfv\in\T_{n}}Z_{\sfv}]$
\item[$\T_{n}^{*}$] population of contaminated cells at time $n$ $[\{\sfv\in\T_{n}:Z_{\sfv}>0\}]$
\item[$\calT_{n}^{*}$] number of contaminated cells at time $n$ $[=\#\T_{n}^{*}]$
\item[$X^{(\bullet,k)}_{i,\sfv}$] given that the cell $\sfv$ has $k$ daughter cells $\sfv 1,...,\sfv k$, the $j^{\,th}$ component $X^{(j,k)}_{i,\sfv}$ of this $\N_{0}^{k}$-valued random vector gives the number of offspring of the $i^{\,th}$ parasite in $\sfv$ which is shared into daughter cell $\sfv j$.
\item[$X^{(\bullet,k)}$] generic copy of the $X^{(\bullet,k)}_{i,\sfv}$, $i\in\N,\,\sfv\in\V$ with components $X^{(j,k)}$
\item[$\mu_{j,k}$] $=\E X^{(j,k)}$
\item[$\gamma$] mean number of offspring per parasite $[=\E\calZ_{1}=\sum_{k\ge 1}p_{k}\sum_{j=1}^{k}\mu_{j,k}]$
\item[$\nu$] mean number of daughter cells per cell $[=\E N=\sum_{k\ge 1}p_{k}]$
\item[$(V_{n})_{n\ge 0}$] infinite random cell line in $\V$ starting at $V_{0}=\varnothing$.
\item[$(Z_{n}')_{n\ge 0}$] associated branching process in iid random environment (ABPRE) $\Lambda=(\Lambda_{n})_{n\ge 0}$ and a copy of $(Z_{V_{n}})_{n\ge 0}$.
\end{description}
\bibliographystyle{abbrv}
\bibliography{StoPro}

\end{document}